\documentclass[preprint, 11pt, english]{elsarticle}
\usepackage{amsmath}
\usepackage{latexsym, amssymb, mathtools}
\usepackage{amsthm}
\usepackage{color}
\usepackage{txfonts}
\usepackage{soul}

\newtheorem{thm}{Theorem}[section] %the resolution could also be [subsection]

\newtheorem{cor}[thm]{Corollary}

\newtheorem{exmpl}[thm]{Example}
\newtheorem{lem}[thm]{Lemma}
\newtheorem{prop}[thm]{Proposition}

\newtheorem{ques}[thm]{Question}

\theoremstyle{definition}
\newtheorem{rem}[thm]{Remark}

\newcommand\operA[2]{{\if!#2!\operatorname{#1}\else{\operatorname{#1}_{#2}^{\phantom{I}}}\fi}} % To be used within Bdefs. Usage: $\operA{N}{K/F}$ produces $N_{K/F}$; $\operA{N}{}$ produces $N$.

\newcommand\Cref[1]{{Corollary~\ref{#1}}}%

\def\tr{{\operatorname{Tr}}}

\def\dim{{\operatorname{dim}}}

\def\norm{{\operatorname{N}}}

\newcommand{\Trace}[1][]{\if!#1!\operatorname{Tr}\else{\operatorname{Tr}_{#1}^{\phantom{I}}}\fi} % Usage: $\Tr[K/F](a)$.

\long\def\forget#1\forgotten{{}} %

\def\({\left(}
\def\){\right)}

\usepackage{tabu}

\newif\iffurther
\furtherfalse
\newif\ifXY % turns XY version on/off
\XYtrue     % Turn it on
\ifXY

\input xy
\input xyidioms.tex
\usepackage{xy}
\xyoption{all} %
\fi % For \ifXY

\usepackage{longtable}
\usepackage{babel}

\journal{Algebra Colloquium}

\begin{document}

\begin{frontmatter}

\title{Spaces with Vanishing Characteristic Coefficients}

\author{Adam Chapman}
\ead{adam1chapman@yahoo.com}
\address{School of Computer Science, Academic College of Tel-Aviv-Yaffo, Rabenu Yeruham St., P.O.B 8401 Yaffo, 6818211, Israel}

\begin{abstract}
We prove that the maximal dimension of a subspace $V$ of the generic tensor product of $m$ symbol algebras of prime degree $p$ with $\tr(v^{p-1})=0$ for all $v\in V$ is $\frac{p^{2m}-1}{p-1}$. The same upper bound is thus obtained for $V$ with $\tr(v)=\tr(v^2)=\dots=\tr(v^{p-1})=0$ for all $v \in V$.
We make use of the fact that for any subset $S$ of $\underbrace{\mathbb{F}_p \times \dots \times \mathbb{F}_p}_{n \ \text{times}}$ of $|S| > \frac{p^{n}-1}{p-1}$, for all $u\in V$ there exist $v,w\in S$ and $k\in [\![0,p-1]\!]$ such that $kv+(p-1-k)w=u$. 
\end{abstract}


\end{frontmatter}

\section{Introduction}

Given a central simple algebra $A$ over a field $F$, there always exists an algebraic extension $K/F$ such that $A \otimes K=M_d(K)$. The number $d$ is the degree of $A$, and $K$ is a splitting field of $A$.
Under this embedding of $A$ into $M_d(K)$, the characteristic polynomial of any element in $A$ can be computed, and its coefficients always live in $F$.
Therefore, we obtain homogeneous polynomial forms $f_1,\dots,f_d : A \rightarrow F$ (``the characteristic coefficients") of degrees $1,\dots,d$ respectively, such that the characteristic polynomial of every given $v \in A$ is $$f(x)=x^d+f_1(v) x^{d-1}+\dots+f_{d-1}(v) x+f_d(v).$$
In particular, $-f_1(v)=\operatorname{Tr}(v)$ is the reduced trace form, and $(-1)^d f_d(v)=\norm(v)$ is the reduced norm form.

In this work we focus on specific central simple algebras, which are tensor products of $m$ symbol algebras of a fixed prime degree $p$ over $F$.
The symbol algebras are known to generate ${_pBr}(F)$ when $\operatorname{char}(F)=p$ (by \cite[Theorem 9.1.4]{GS}) or when $\operatorname{char}(F)\neq p$ and $F$ contains primitive $p$th roots of unity (by \cite{MS}), and thus their different tensor products represent all the classes of this group.
Our aim is to understand the structure of $F$-linear subspaces where some of the characteristic coefficients vanish, in particular the first $p-1$ coefficients, $f_1,...,f_{p-1}$.
In the case of $m=1$, these spaces (known as ``Kummer spaces" or ``$p$-central spaces") have strong connections to the Severi-Brauer variety of the algebra by \cite{Matzri:2020}, and their dimension in the generic case was bounded from above by $p+1$ in \cite{CGMRV} and \cite{ChapmanChapman:2017}.
For the special cases of $p=2$ or 3, the upper bound $p+1$ holds for Kummer spaces in all symbol algebras of degree $p$, not merely the generic ones (see \cite{MatzriVishne:2012} and \cite{MatzriVishne:2014} for $p=3$; the case of $p=2$ is classical).
Here, we prove that in the generic case (where the slots are independent transcendental elements over a smaller field), the dimension of an $F$-subspace $V$ of a tensor product of $m$ symbol algebras of degree $p$ with $\tr(v^{p-1})=0$ for all $v\in V$ is bounded from above by $\frac{p^{2m}-1}{p-1}$. Thus, the same upper bound is obtained also for a space over which $f_1,\dots,f_{p-1}$ vanish.
The proof makes use of the fact that for any subset $S$ of $\mathbb{F}_p^{\times n}=\underbrace{\mathbb{F}_p \times \dots \times \mathbb{F}_p}_{n \ \text{times}}$ of $|S|>\frac{p^n-1}{p-1}$, every $u \in V$ can be written as $kv+(p-1-k)w$ for some $v,w \in S$ and $k \in [\![0,p-1]\!]$.

The generic case is of special significance because every tensor product of symbol algebras is a specialization of the generic tensor product, and heuristically, statements that hold true in the generic case are expected to hold true in ``most" of the cases. Spaces over which certain characteristic coefficients vanish in the generic case give rise to spaces with the same property in every specialization, and thus the generic case contains the ``smallest" such spaces, and therefore it is the right place to start if one wishes to examine if a certain hypothesis on the dimension holds.

\section{Lines in $\mathbb{F}_p^{\times n}$}\label{LinesInFp}

By a line in $\mathbb{F}_p^{\times n}$ we mean an affine line, i.e., a set of the form $\{v+(w-v)t : t\in \mathbb{F}_p\}$ for some $v\in \mathbb{F}_p^{\times n}$ and $w\in \mathbb{F}_p^{\times n} \setminus \{v\}$.
\begin{lem}
Given a prime $p$, $n \in \mathbb{N}$ and a set $S \in \mathbb{F}_p^{\times n}$ of $|S|>\frac{p^n-1}{p-1}$, every point $P \in V$ lives on a line passing through two distinct points of $S$.
\end{lem}

\begin{proof}
There are exactly $\frac{p^n-1}{p-1}$ different lines passing through any given point $P$ in $\mathbb{F}_p^{\times n}$.
By the pigeon-hole principle, since $|S| \geq \frac{p^n-1}{p-1}+1$, there exists a line passing through $P$ that contains two distinct points from $S$.
\end{proof}

\begin{prop}
Given a prime $p$, $n \in \mathbb{N}$ and a set $S \subseteq \mathbb{F}_p^{\times n}$ of $|S|>\frac{p^n-1}{p-1}$, every point $Q \in \mathbb{F}_p^{\times n}$ can be written as $kv+(p-1-k)w$ for some $v,w \in S$ and $k\in [\![0,p-1]\!]$.
\end{prop}

\begin{proof}
Take $P=-Q$ in the previous lemma.
And then there exist $v,w\in S$ such that $-Q$ lives in the line passing through $v$ and $w$.
Every point on this line is either $v$, or $w$ or $kv+(p+1-k)w$ for some $k \in [\![2,p-1]\!]$.
Therefore, $Q$ is either $-v=(p-1)v$, $-w=(p-1)w$, or $(p-k)v+(k-1)w$ for some $k \in [\![2,p-1]\!]$.
\end{proof}

\section{Symbol Algebras}

Fix a prime number $p$. 
A symbol algebra of degree $p$ over a field $F$ is defined in the following two cases:
\begin{enumerate}
\item $\operatorname{char}(F)=p$, and
\item $\operatorname{char}(F) \neq p$ and $F$ contains a primitive $p$th root of unity $\rho$.
\end{enumerate}
In both cases, they coincide with cyclic algebras, but general cyclic algebras of degree $p$ exist also in cases where $\operatorname{char}(F) \neq p$ and $F$ contains no primitive $p$th root of unity.

When $\operatorname{char}(F)=p$, a symbol algebra takes the form
$A=[\alpha,\beta)_{p,F}=F \langle x,y : x^p-x=\alpha, y^p=\beta, y x y^{-1}=x+1 \rangle$ for some $\alpha \in F$ and $\beta \in F^\times$. In this case, $$\operatorname{Tr}(x^i y^j)=\begin{cases} -1 & (i,j)=(p-1,0)\\ 0 & (i,j) \in [\![0,p-1]\!]^{\times 2} \setminus \{(p-1,0)\}.
\end{cases}$$

When $\operatorname{char}(F)\neq p$ and $F$ contains a primitive $p$th root of unity $\rho$, a symbol algebra takes the form
$A=(\alpha,\beta)_{p,F}=F \langle x,y : x^p=\alpha, y^p=\beta, y x y^{-1}=\rho x \rangle$ for some $\alpha,\beta \in F^\times$. In this case, $\operatorname{Tr}(x^i y^j)=0 \forall(i,j) \in [\![0,p-1]\!]^{\times 2} \setminus \{\vec{0}\}$, and $\operatorname{Tr}(1)=p$.

Note that for any two central simple algebras $A$ and $B$ over $F$, in $A\otimes B$ we have $\operatorname{Tr}(a \cdot b)=\operatorname{Tr}(a) \cdot \operatorname{Tr}(b)$ for $a\in A$ and $b\in B$. Therefore, for $(d_1,e_1,\dots,d_m,e_m) \in [\![0,p-1]\!]^{\times 2m}$, the trace form on the tensor product $[\alpha_1,\beta_1)_{p,F} \otimes \dots \otimes [\alpha_m,\beta_m)_{p,F}$ with generators $x_1,y_1,\dots,x_m,y_m$ satisfies $\operatorname{Tr}(x_1^{d_1} y_1^{e_1} \dots x_m^{d_m} y_m^{e_m})=(-1)^m$ if and only if $d_1=\dots=d_m=p-1$ and $e_1=\dots=e_m=0$, and $\operatorname{Tr}(x_1^{d_1} y_1^{e_1} \dots x_m^{d_m} y_m^{e_m})=0$ otherwise.
In the tensor product $(\alpha_1,\beta_1)_{p,F}\otimes \dots \otimes (\alpha_m,\beta_m)_{p,F}$ with generators $x_1,y_1,\dots,x_m,y_m$, the trace form satisfies $\operatorname{Tr}(1)=p^m$ and $\operatorname{Tr}(x_1^{d_1}y_1^{e_1}\dots x_m^{d_m} y_m^{e_m})=0$ whenever $(d_1,e_1,\dots,d_m,e_m) \in [\![0,p-1]\!]^{\times 2m} \setminus \{ \vec{0}\}$.

\section{Trace conditions for vanishing characteristic coefficients}

Let $v_1,...,v_k$ be elements in some non-commutative algebra and $d_1,...,d_k$ be non-negative integers. The notation $v_1^{d_1} * \dots * v_{k}^{d_{k}}$ stands for the sum of all the different products of $d_1$ copies of $v_1$, $d_2$ copies of $v_2$ and so on (see \cite{Revoy:1977}). For example, $v_1^2 * v_2=v_1^2 v_2+v_1 v_2 v_1+v_2 v_1^2$.

\begin{prop}
Let $p$ be a prime number, $F$ a field of $\operatorname{char}(F)=0$ or $\geq p$, $A$ a central simple $F$-algebra of degree $d \geq p$, and $V$ an $F$-vector subspace of $A$.
Let $f_1,\dots,f_d$ be the characteristic coefficients of $A$, as in the introduction.
Then for any $r \in [\![1,p-1]\!]$, 
we have $f_1(v)=\dots=f_r(v)=0$ for any $v\in V$ if and only if $\operatorname{Tr}(v)=\dots=\operatorname{Tr}(v^r)=0$ for any $v\in V$.
\end{prop}

\begin{proof}
The argument follows the same lines as {\cite[Theorem 31]{MRSV}}.
The only reason why in the original theorem the statement includes only one direction is because of the issues with the characteristic, but under our assumptions, we avoid these issues.
\end{proof}
\begin{rem}
Write $V=F v_1+\dots+F v_\ell$.
Then the condition $\operatorname{Tr}(v^r)=0$ for all $v\in V$ boils down to $\operatorname{Tr}(v_1^{d_1} * \dots * v_\ell^{d_\ell})=0$ for any non-negative integers $d_1,\dots,d_\ell$ with $d_1+\dots+d_\ell = r$.
As a result, the condition $\operatorname{Tr}(v^r)=0$ for all $v\in V$ is invariant under scalar extension.\end{rem}

\section{Bounding the dimension}

Throughout this section, assume $p$ is a prime integer and $F_0$ is a field of characteristic either 0 or $\geq p$. If $\operatorname{char}(F_0) \neq p$, assume $F_0$ contains a primitive $p$th root of unity $\rho$.

\begin{thm}
Let $A=(\alpha_1,\beta_1)_{p,F} \otimes \dots \otimes (\alpha_m,\beta_m)_{p,F}$ if $\operatorname{char}(F_0)\neq p$, and $A=[\alpha_1,\beta_1)_{p,F} \otimes \dots \otimes [\alpha_m,\beta_m)_{p,F}$ if $\operatorname{char}(F_0)=p$ over $F=F_0(\alpha_1,\beta_1,\dots,\alpha_m,\beta_m)$ or $F_0(\!(\alpha_1^{-1})\!)(\!(\beta_1^{-1})\!)\dots(\!(\alpha_m^{-1})\!)(\!(\beta_m^{-1})\!)$.
Then every space $V$ in $A$ satisfying $\operatorname{Tr}(v^{p-1})=0$ for all $v\in V$ is of dimension at most $\frac{p^{2m}-1}{p-1}$. 
\end{thm}

\begin{proof}
Since the condition $\operatorname{Tr}(v^{p-1})=0$ for all $v\in V$ is invariant to scalar extensions, consider the case of iterated Laurent series $F=F_0(\!(\alpha_1^{-1})\!)(\!(\beta_1^{-1})\!)\dots(\!(\alpha_m^{-1})\!)(\!(\beta_m^{-1})\!)$.
In this case, $A$ is also a division ring of iterated twisted Laurent series, see \cite[Section 1.2.6]{TignolWadsworth:2015}.
The space $W$ of leading terms of elements in $V$ is of the same dimension as $V$, spanned by leading terms of values which are distinct modulo $\mathbb{Z}^{\times 2m}$ (see \cite[Remark 2.2]{ChapmanUre:2017} for a detailed explanation).
Set $x_1,y_1,\dots,x_m,y_m$ the generators of $A$.
Then, $W$ is generated by elements of the form $\overrightarrow{xy}^{\overrightarrow{de}}=x_1^{d_1}y_1^{e_1}\dots x_m^{d_m}y_m^{e_m}$ where $\overrightarrow{de}=(d_1,e_1,\dots,d_m,e_m)$ are distinct elements in $[\![0,p-1]\!]^{\times 2m}$.
Set $S$ to be the set of those $\overrightarrow{de}$.
Suppose the contrary that $\dim V=\dim W =|S|\geq \frac{p^{2m}-1}{p-1}+1$.
If $\operatorname{char}(F_0)\neq p$ take $P=\vec{0}$, and if $\operatorname{char}(F_0)=p$ take $P=(p-1,0,p-1,0,\dots,p-1,0)$. In either case, $P$ is equivalent modulo $p\mathbb{Z}^{\times 2m}$ to $k v+(p-1-k)w$ for some $v,w \in S$ and $k \in [\![0,p-1]\!]$.
Now, the trace of the leading term of $v^k * w^{p-1-k}$ is a nonzero scalar times the trace of $\vec{xy}^P$, which is nonzero. Therefore, $\operatorname{Tr}(v^k * w^{p-1-k}) \neq 0$, contradiction.
\end{proof}

\begin{rem}
Note that by this theorem, the maximal dimension of a subspace satisfying $\operatorname{Tr}(v^{p-1})=0$ is $p+1$ when $m=1$, which is a stronger result than the results of \cite{CGMRV} and \cite{ChapmanChapman:2017} where the maximal dimension of a ``Kummer space" (i.e., a space in which $f_1(v)=\dots=f_{p-1}(v)=0$) is $p+1$.
\end{rem}

\begin{cor}
The dimension of a subspace $V$ of $A$ satisfying $f_1(v)=\dots=f_{p-1}(v)=0$ for all $v\in V$ is bounded from above by $\frac{p^{2m}-1}{p-1}$.
\end{cor}

It is natural to ask if this bound is sharp, and indeed, we can produce an example:
\begin{exmpl}
Consider $V=Fx_1+F[x_1]y_1+F[x_1,y_1]x_2+F[x_1,y_1
,x_2]y_2+\dots+F[x_1,y_1,...,x_{m-1},y_{m-1}]x_m+
F[x_1,y_1,...,x_{m-1},y_{m-1},x_m]y_m$ when $\operatorname{char}(F_0)\neq p$. This space satisfies $\operatorname{Tr}(v)=\dots=\operatorname{Tr}(v^{p-1})=0$ for all $v\in V$, and its dimension is indeed $1+p+\dots+p^{2m-1}=\frac{p^{2m}-1}{p-1}$.
\end{exmpl}

\begin{proof}
For convenience, write $z_1,z_2,\dots,z_{2m-1},z_{2m}$ for $x_1,y_1,\dots,x_m,y_m$. Then we prove by induction that $F[z_1]+F[z_1]z_2+\dots+F[z_1,\dots,z_{k-1}]z_k$ satisfies the statement for each $k \in \{1,\dots,2m\}$. We do it by induction on $k$. Recall that $\tr(z_1^{d_1} z_2^{d_2}\dots z_k^{d_k})\neq 0$ if and only if $d_1 \equiv d_2 \equiv \dots \equiv d_k \equiv 0 \pmod{p}$. For $k=1$, $Fz_1$ is a space in which every element $v=cz_1$ clearly satisfies $\tr(v^i)=\tr(c^iz_1^i)=0$ for each $i\in [\![1,p-1]\!]$. Suppose the statement holds true for $k-1$. Then consider $v\in F[z_1]+F[z_1]z_2+\dots+F[z_1,\dots,z_{k-2}]z_{k-1}+F[z_1,\dots,z_{k-1}]z_k$. It decomposes as $u+w$ where $u\in F[z_1]+F[z_1]z_2+\dots+F[z_1,\dots,z_{k-2}]z_{k-1}$ and $w\in F[z_1,\dots,z_{k-1}]z_k$. For each $i \in [\![1,p-1]\!]$, $\tr((u+w)^i)=\sum_{d=0}^i\tr(u^d * w^{i-d})$. The term $u^i$ has trace zero by the induction hypothesis. In each summand of $u^d * w^{i-d}$ with $d\neq 0$, the power of $z_k$ is $i-d$ which is an integer in $[\![1,p-1]\!]$, and thus $\tr(u^d * w^{i-d})=0$. 
\end{proof}

Another natural question is to ask whether there exists a space $V$ satisfying $\operatorname{Tr}(v^{p-1})=0$ for all $v \in V$, but not $f_1(v)=\dots=f_{p-1}(v)=0$. The answer is again positive:
\begin{exmpl}
Suppose $\operatorname{char}(F_0)=3$ and that $F_0$ contains a square root $\omega$ for $-1$.
Take $m=1$, i.e., $A=[\alpha,\beta)_{3,F}$ generated by $x$ and $y$ over $F=F_0(\!(\alpha^{-1})\!)(\!(\beta^{-1})\!)$.
Consider the space $V=F[x]y+F(x^2-\omega x)$.
Its basis is $y,xy,x^2y,x^2-\omega x$.
Clearly $\operatorname{Tr}(v^2)=\operatorname{Tr}(v*w)=0$ whenever $v \in \{y,xy,x^2y\}$ and $w \in \{y,xy,x^2y,x^2-\omega x\} \setminus \{v\}$.
It is left to compute $\operatorname{Tr}((x^2-\omega x)^2)$, which is
$\operatorname{Tr}(x^4-2\omega x^3-x^2)=\operatorname{Tr}((x+\alpha)x-2\omega (x+\alpha)-x^2)=\operatorname{Tr}((\alpha-2\omega)x-2\omega \alpha)=0$.
However, $\operatorname{Tr}(x^2-\omega x)=-1 \neq 0$.
\end{exmpl}

This example was possible because here $F \not \subseteq V$. When $F \subseteq V$, this is impossible:

\begin{prop}
If $\operatorname{char}(F)=p$, $A=[\alpha_1,\beta_1)_{p,F} \otimes \dots \otimes [\alpha_m,\beta_m)_{p,F}$ and $V$ is an $F$-subspace of $A$ satisfying $\operatorname{Tr}(v^{p-1})=0$ for all $v \in V$ containing $F$ as a subspace,  then $f_1(v)=\dots=f_{p-1}(v)=0$ for all $v \in V$.
\end{prop}

\begin{proof}
Set $v_1,v_2,\dots,v_k$ to be a basis of $V$.
The condition $\operatorname{Tr}(v^{p-1})=0$ is equivalent to $\operatorname{Tr}(v_1^{d_1} *v_2^{d_2} *\dots * v_k^{d_k})=0$ for any non-negative $d_1,\dots,d_k$ with $d_1+\dots+d_k=p-1$.
We can assume $v_1=1$ because $F$ is a subspace of $V$.
Then $\operatorname{Tr}(v_1^{d_1} *v_2^{d_2} *\dots * v_k^{d_k})=\binom{p-1}{d_1}\operatorname{Tr}(v_2^{d_2} *\dots * v_k^{d_k})$. Therefore, $\operatorname{Tr}(v_2^{d_2} *\dots * v_k^{d_k})=0$ for any non-negative $d_2,\dots,d_k$ with $d_2+\dots+d_k=\ell \leq p-1$. Now, for any $r\in \{0,1,\dots,p-1-\ell\}$, $\operatorname{Tr}(v_1^{r} *v_2^{d_2} *\dots * v_k^{d_k})=\binom{r+\ell}{r}\operatorname{Tr}(v_2^{d_2} *\dots * v_k^{d_k})=0$, and so $\operatorname{Tr}(v^{t})=0$ for any $v\in V$ and any $t \in [\![1,p-1]\!]$. Therefore, $f_1(v)=\dots=f_{p-1}(v)=0$ for all $v \in V$.
\end{proof}

\section{The trace of the square and open questions}

In this section we want to say a few words on the condition $\operatorname{Tr}(v^2)=0$ over general central simple algebras of degree $p^m$.

\begin{prop}
Given an odd prime $p$, a field $F$ of $\operatorname{char}(F)=0$ or $\geq p$ and a central simple algebra $A$ of degree $p^m$ over $F$, the form $\varphi : A \rightarrow F$ defined by $\varphi(v)=\operatorname{Tr}(v^2)$ is a non-singular quadratic form, and thus the maximal dimension of a subspace $V$ of $A$ satisfying $\operatorname{Tr}(v^2)=0$ for all $v \in V$ is at most $\frac{p^{2m}-1}{2}$.
\end{prop}

\begin{proof}
Take $K$ to be an algebraic closure of $F$, and consider $A \otimes K \cong M_{p^m}(K)$. The form $\varphi$ restricts accordingly to $\varphi_K : A\otimes K \rightarrow K$.
Singularity is invariant under scalar extension, and thus it is enough to explain why $\varphi_K$ is non-singular.
Now, the form $\varphi_K$ is isometric to $\langle 1,1,\dots,1 \rangle=p^m \times \langle 1 \rangle$ by choosing the following basis:
the $p^m$ diagonal matrices $E_{1,1},\dots,E_{p^m,p^m}$ and the $p^{2m}-p^m$ matrices of the forms $\frac{1}{\sqrt{2}}(E_{i,j}+E_{j,i})$ and $\frac{1}{\sqrt{2}}(\sqrt{-1}E_{i,j}-\sqrt{-1}E_{j,i})$ with $i\neq j$, where $E_{i,j}$ stands for the matrix unit with 1 in the $(i,j)$th slot and 0 elsewhere. It is therefore nonsingular.
\end{proof}

One can ask in this direction whether for any degree $p^m$ central simple $F$-algebra $A$, the maximal dimension of an $F$-subspace $V$ of $A$ over which $\operatorname{Tr}(v^r)=0$ for a given $r \in [\![2,p-1]\!]$ is $\lfloor \frac{p^{2m}-1}{r} \rfloor$.
The answer in this full generality is negative, for if we take $A$ to be $M_{p^m}(F)$ and $V$ to be the space of upper triangular matrices with zeros on the diagonal, then this satisfies $\operatorname{Tr}(v^r)=0$ for any $v \in V$ and $r\in [\![1,p-1]\!]$ and $\dim V=(p^m-1)+(p^m-2)+\dots+2+1=\frac{p^m\cdot (p^m-1)}{2}=\frac{p^{2m}-p^m}{2}$, which for large $p$ is much greater than $\lfloor \frac{p^{2m}-1}{r} \rfloor$ for $r\geq 3$.
In the generic tensor product of $m$ symbol algebras of degree $p$ with $r=p-1$ it holds true, so at the very least we can ask whether this holds for division algebras:
\begin{ques}
Is $\lfloor \frac{p^{2m}-1}{r} \rfloor$ an upper bound for the dimension of subspaces $V$ of any division algebra $A$ of degree $p^m$ where $\operatorname{Tr}(v^r)=0$ for all $v\in V$ given $r\in [\![2,p-1]\!]$? Is it true at least when $A$ is the generic tensor product of $m$ symbol algebras of degree $p$?
\end{ques}

\section*{Acknowledgments}

The author thanks Terence Tao and Gjergji Zaimi for their answers to my question on MathOverflow (https://mathoverflow.net/questions/453198/subsets-of-mathbbz-p-times-n).
The author also thanks Ian Le, Ryan Alweiss and Sammy Luo for the conversations that helped formulating the combinatorial problem.
Special thanks go to the two anonymous referees for the detailed and helpful reports.
%\section*{Bibliography}
\bibliographystyle{amsalpha}
\newcommand{\etalchar}[1]{$^{#1}$}
\def\cprime{$'$} \def\cprime{$'$}
\providecommand{\bysame}{\leavevmode\hbox to3em{\hrulefill}\thinspace}
\providecommand{\MR}{\relax\ifhmode\unskip\space\fi MR }
% \MRhref is called by the amsart/book/proc definition of \MR.
\providecommand{\MRhref}[2]{%
  \href{http://www.ams.org/mathscinet-getitem?mr=#1}{#2}
}
\providecommand{\href}[2]{#2}

\end{document}